\documentclass[a4paper,11pt]{article}
\usepackage{amsmath}
\usepackage{amssymb}
\usepackage{theorem}
\usepackage{pstricks,url}
\usepackage{euscript}
\usepackage{epic,eepic}
\usepackage{graphicx}
\PassOptionsToPackage{normalem}{ulem}
\usepackage{amsmath}
\usepackage{amssymb}
\usepackage{theorem}
\usepackage{pstricks}
\usepackage{euscript}
\usepackage{epic,eepic}
\usepackage{graphicx}
\usepackage{setspace}
\PassOptionsToPackage{normalem}{ulem}
\newcommand{\minimize}[2]{\ensuremath{\underset{\substack{{#1}}}%
{\mathrm{minimize}}\;\;#2 }}

\newcommand{\menge}[2]{\big\{{#1} \;|\; {#2}\big\}}

\newcommand{\emp}{\ensuremath{{\varnothing}}}

\newcommand{\scal}[2]{\left\langle{#1}\mid {#2} \right\rangle}

\newcommand{\HH}{\ensuremath{\mathcal H}}
\newcommand{\GG}{\ensuremath{\mathcal G}}

\newcommand{\BL}{\ensuremath{\EuScript B}\,}

\newcommand{\HHH}{\ensuremath{\boldsymbol{\mathcal H}}}

\newcommand{\GGG}{\ensuremath{\boldsymbol{\mathcal G}}}
\newcommand{\VV}{\ensuremath{\boldsymbol{V}}}
\newcommand{\CCC}{\ensuremath{\boldsymbol{C}}}

\newcommand{\sri}{\ensuremath{\operatorname{sri}}}

\newcommand{\RR}{\ensuremath{\mathbb R}}

\newcommand{\RP}{\ensuremath{\left[0,+\infty\right[}}
\newcommand{\RPP}{\ensuremath{\,\left]0,+\infty\right[}}

\newcommand{\NN}{\ensuremath{\mathbb N}}

\newcommand{\dom}{\ensuremath{\operatorname{dom}}}

\newcommand{\prox}{\ensuremath{\operatorname{prox}}}

\newcommand{\cart}{\ensuremath{\mbox{\Large{$\times$}}}}

\newcommand{\argmin}{\ensuremath{\operatorname{argmin}}}
\newcommand{\ran}{\ensuremath{\operatorname{ran}}}
\newcommand{\zer}{\ensuremath{\operatorname{zer}}}
\newcommand{\gra}{\ensuremath{\operatorname{gra}}}

\newcommand{\vv}{\ensuremath{\boldsymbol{v}}}

\newcommand{\xx}{\ensuremath{\boldsymbol{x}}}
\newcommand{\pp}{\ensuremath{\boldsymbol{p}}}

\newcommand{\yy}{\ensuremath{\boldsymbol{y}}}

\newcommand{\rr}{\ensuremath{\boldsymbol{r}}}

\newcommand{\ww}{\ensuremath{\boldsymbol{w}}}

\newcommand{\LL}{\ensuremath{\boldsymbol{L}}}
\newcommand{\PPP}{\ensuremath{\mathsf{P}}}
\newcommand{\UU}{\ensuremath{\boldsymbol{U}}}

\newcommand{\E}{\ensuremath{\mathsf{E}}}

\newcommand{\AAA}{\ensuremath{\boldsymbol{A}}}

\newcommand{\FF}{\ensuremath{{\EuScript F}}}

\newcommand{\Id}{\ensuremath{\operatorname{Id}}}

\newcommand{\weakly}{\ensuremath{\rightharpoonup}}

\newtheorem{theorem}{Theorem}[section]
\newtheorem{lemma}[theorem]{Lemma}

\newtheorem{corollary}[theorem]{Corollary}

\theoremstyle{plain}{\theorembodyfont{\rmfamily}
}
\theoremstyle{plain}{\theorembodyfont{\rmfamily}
}
\theoremstyle{plain}{\theorembodyfont{\rmfamily}
\newtheorem{algorithm}[theorem]{Algorithm}}
\theoremstyle{plain}{\theorembodyfont{\rmfamily}
}
\theoremstyle{plain}{\theorembodyfont{\rmfamily}
\newtheorem{problem}[theorem]{Problem}}
\theoremstyle{plain}{\theorembodyfont{\rmfamily}
\newtheorem{remark}[theorem]{Remark}}
\theoremstyle{plain}{\theorembodyfont{\rmfamily}
}

\definecolor{labelkey}{rgb}{0,0.08,0.45}
\definecolor{refkey}{rgb}{0,0.6,0.0}
\definecolor{Brown}{rgb}{0.45,0.0,0.05}
\definecolor{dgreen}{rgb}{0.00,0.49,0.00}
\definecolor{dblue}{rgb}{0,0.08,0.75}
\numberwithin{equation}{section}

\begin{document}
\title{\sffamily\huge 
A first-order stochastic primal-dual algorithm with correction step
}
\author{Lorenzo Rosasco$^{1,2}$, Silvia Villa$^1$ and  B$\grave{\text{\u{a}}}$ng C\^ong V\~u$^1$\\ 
\small$^1$ LCSL, Istituto Italiano di Tecnologia and Massachusetts Institute of Technology\\
\small Bldg. 46-5155, 77 Massachusetts Avenue, Cambridge, MA 02139, USA\\
\small $^2$ DIBRIS, Universit\`a di Genova, Via Dodecaneso 35\\ 
\small 16146 Genova, Italy\\
\url{lrosasco@mit.edu};\; \url{silvia.villa@iit.it};\;\url{cong.bang@iit.it}
}
\maketitle
\begin{abstract}
We investigate the convergence properties of a stochastic primal-dual splitting
algorithm for solving structured monotone inclusions involving the sum of a cocoercive operator 
and  a composite monotone operator. The proposed method is the stochastic extension to 
monotone inclusions of a proximal method studied in \cite{DroSabTeb15, LorVer11} 
for saddle point problems. It consists in a forward step 
determined by the stochastic evaluation of the cocoercive operator, a backward step in the dual variables
involving the resolvent of the monotone operator, and an additional forward step using the  stochastic evaluation
of the cocoercive introduced in the first step.  
We prove weak  almost sure convergence of the iterates by showing that the primal-dual sequence
generated by the method is stochastic quasi Fej\'er-monotone with respect to the set of zeros
of the considered primal and dual inclusions. 
Additional results on ergodic convergence in expectation are considered for the 
special case of saddle point models. 
\end{abstract}

 

\section{Introduction}
This paper is concerned with the algorithmic solution, in a stochastic setting, of 
structured monotone inclusions defined by the sum of a cocoercive operator and 
a monotone operator composed with a linear transformation and its adjoint. 
This  problem arises in many applications,
such as  variational inequalities and equilibrium problems \cite{Facc03}, signal and image processing \cite{siam05,Dau04}, 
game theory \cite{Bric13}, and statistical learning \cite{Devi11,Duch09,MRSVV10,RosVilMos13,VilSal13}.
A necessarily incomplete list of related works include
 \cite{BotCseHei15,siop2,ChaPoc11,
Cham14,ComPes12,30Combettes13,Con13,Davis14,PesKom15,aicm1}.

For monotone inclusions with the considered  structure, key is the joint solution of the primal problem and its associated dual form. 
Indeed, primal-dual schemes have several advantages, since they do not require the inversion of any of the
involved linear functions, and independently {\em activate} each of the monotone operators. 
The single-valued operators is involved in the forward step, while  the set-valued operator  appears in the  backward steps and  
requires the computation of the resolvent \cite{ComPes12,aicm1}.  \\
Recently, stochastic versions of splitting methods for monotone inclusions have been studied. This is relevant to
consider practical situations where operators 
are known only through measurements subject to random noise, or when 
 the  computation of a stochastic estimate is cheaper than the evaluation of the operator itself.   
Among the many approaches, we mention
stochastic forward-backward splitting \cite{plc14,LSB14b, BiaHac15},  stochastic Douglas-Rachford \cite{plc14}, and 
stochastic versions of primal-dual methods as in \cite{BiaHacIut14,plc14,pesquet14, RosVilVu15}.
These works  found  natural applications in stochastic optimization \cite{plc14,LSB14b} and machine learning  \cite{Duch09,LSB14a}.

In this paper we focus on a primal dual splitting method, that generalizes to monotone inclusions
the so called proximal alternating predictor-corrector algorithm (PAPC), proposed independently in \cite{LorVer11} for
regularized least squares minimization problems and in \cite{DroSabTeb15} for saddle point problems. 
In particular,  we allow for stochastic errors: the proposed algorithm requires the computation 
of the resolvent of the maximal monotone operators at each step, and 
it uses a stochastic approximation of the cocoercive operator. The update of the primal variable requires an extra step 
(called a correction step in \cite{DroSabTeb15}). In this respect, 
the algorithm resembles the extragradient method proposed in \cite{Nem04}, but differs from it since 
 it does not require extra evaluations of the involved operators. Indeed, 
the correction step does not impact the complexity of the method since it requires only matrix/vector multiplications. 
The considered algorithm is also related to the general inertial splitting scheme \cite[Algorithm 4.41]{RosVilVu15}
for solving multivariate monotone inclusions. In contrast to \cite[Algorithm 4.41]{RosVilVu15}, 
we allow for an additional projection step on a set of linear constraints, and  consider variable step-sizes. 
Our analysis establishes almost sure weak convergence of the iterates generated by the method, and ergodic convergence 
under more general conditions on the error when saddle points problems are considered. 
The analysis differs from the one in \cite{DroSabTeb15,LorVer11}, and  is based on variable metric
 stochastic quasi-Fejer sequences \cite{Vu15}.   
 
The paper is organized as follows: in Section~\ref{backgr} we introduce notation and preliminaries. 
We then  present the problem, the algorithm, and we prove almost sure convergence in Section~\ref{sec:main}.
In Section~\ref{sec:erg} we focus on minimization problems and their saddle point formulations. 
For this special case, we prove ergodic convergence in expectation of the duality gap under more general conditions on the errors
and the step-sizes. 
Finally, in Section~\ref{sec:appl}, we consider the case of sum of composite inclusions and we show how to apply the general scheme
to this case, using the product space reformulation \cite{siop2}. As a corollary, we obtain convergence 
results for structured minimization problems.

\begin{problem} 
\label{p:prob}
Let $\beta\in\left]0,+\infty\right[$, let $\HH$ and $\GG$ be real Hilbert spaces, let $B\colon\HH\to\HH$ 
be a $\beta$-cocoercive operator, let $V$ be a closed vector subspace of $\HH$, let 
$A\colon \GG\to 2^{\GG}$ be a maximally monotone operator,  and let $L\colon\HH\to\GG$ be a bounded linear operator. 
The normal cone operator to $V$ is denoted by $N_V$.
Denote by $\mathcal{P}$ the set of all points $\overline{x}\in\HH$ such that 
\begin{equation}
0 \in B\overline{x} + L^* A(L\overline{x}) + N_V\overline{x},
\end{equation}  
and $\mathcal{D}$ the set of all $\overline{v}\in\GG$ such that 
 \begin{equation}
0  \in -L(B+N_V)^{-1}(-L^*\overline{v}) + A^{-1}\overline{v},
\end{equation}  
The problem is to find a point $(\overline{x},\overline{v})$ in $\mathcal{P}\times\mathcal{D}$.
\end{problem}
Note that since $V$ is a closed vector subspace,
\[
 N_Vx=
\begin{cases}
V^\perp & \text{if } x\in V\\
\varnothing & \text{otherwise } .
\end{cases} 
\]
\section{Notation and preliminary results}
\label{backgr}
Throughout, $\HH$  is a real separable Hilbert space.
We denote by  $\scal{\cdot}{\cdot}$ and $\|\cdot\|$  the scalar product and the associated norm
of $\HH$.  
The symbols $\weakly$ and $\to$ denote weak and strong convergence, respectively. 
We denote  by $\ell_+^1(\NN)$ the set of summable sequences in  $\RP$, and 
by $\BL(\HH)$ the space of linear operators from $\HH$ into itself. 
Let $U\in\BL(\HH)$ be self-adjoint and strongly positive, i.e. 
\begin{equation}
(\exists \chi\in\RPP)(\forall x\in\HH) \quad \scal{Ux}{x}\geq \chi\|x\|^2.
\end{equation} 
We define a scalar product and a norm respectively by
\[(\forall x\in\HH)(\forall y\in\HH)\quad\scal{x}{y}_U=\scal{Ux}{y}
\quad\text{and}\quad\|x\|_U=\sqrt{\scal{Ux}{x}}.
\]
Let $A\colon\HH \to 2^\HH$ be a set-valued operator.
The domain and the graph of $A$ are defined by  
$$\dom A =\menge{x\in\HH}{Ax\neq\emp}\;\text{and}\;
\gra A =\menge{(x,u)\in\HH\times\HH}{u\in Ax}.$$
The set of zeros of $A$ is $\zer A=\menge{x\in\HH}{0\in Ax}$
and the range of $A$ is  $\ran A=A(\HH)$.
The inverse of $A$  is $A^{-1}\colon\HH \to 2^\HH\colon u\mapsto 
\menge{x\in\HH}{u\in Ax}$. The resolvent of $A$ is 
\begin{equation}
\label{eq:res}
J_A=(\Id+A)^{-1},
\end{equation}
where $\Id$ denotes the identity operator of $\HH$. Moreover, 
$A$ is monotone if
\[
(\forall (x,u)\in\gra A) (\forall (y,v)\in\gra A)\quad \scal{x-y}{u-v} \geq 0, 
\]
and maximally so, if there exists no monotone operator $\widetilde{A}\colon\HH\to\HH$
such that $\gra A\subset \gra \widetilde{A}\neq \gra A$.
If  $A$ is monotone, then $J_A$ is single-valued and nonexpansive, and, 
in addition, if $A$ is maximally monotone, then $\dom J_A=\HH$.

 Let $\Gamma_0(\HH)$ be the class of proper lower semicontinuous convex functions from $\HH$ to
$\left]-\infty,+\infty\right]$. For any self-adjoint strongly positive operator
$U\in \BL(\HH)$ and $f\in\Gamma_0(\HH)$, we define
 \begin{equation}
\label{set0}
\prox_{f}^{U}\colon\HH\to\HH \colon x\mapsto\underset{y\in\HH}{\argmin}\: 
\big( f(y) + \frac{1}{2}\|x-y\|_{U}^2\big),
\end{equation}
and 
 \[
\prox_{f}\colon\HH\to\HH \colon x\mapsto\underset{y\in\HH}{\argmin}\: 
\big( f(y) + \frac{1}{2}\|x-y\|^2\big).
\]
It holds $\prox_{f}^U=J_{U^{-1}\partial f}$, and $\prox_{f}= J_{\partial f}$
coincides with the classical definition of proximity operator in \cite{Mor62}.
Moreover, let $x\in\HH$ and set $p=\prox_f^U x$. Then
\begin{equation}
\label{e:up}
(\forall y\in\HH)\quad f(p)-f(y) \leq \scal{y-p}{U(p-x)}.
\end{equation}
The  conjugate function of $f$ is 
\[
 f^*\colon a\mapsto \sup_{x\in\HH}\big(\scal{a}{x}-f(x)\big).
\]
Note that, 
\[
 (\forall f\in\Gamma_0(\HH))(\forall (x,y)\in\HH^2)\quad y\in\partial f(x) \Leftrightarrow x\in\partial f^*(y),
\]
or equivalently,
\begin{equation}
\label{e:equi}
 (\forall f\in\Gamma_0(\HH))\quad (\partial f)^{-1} = \partial f^*.
\end{equation}
 The strong relative interior of a subset $C$ of $\HH$, denoted by $\sri C$, is 
the set of points $x\in C$ such that the cone generated by 
$-x + C$ is a closed vector subspace of $\HH$. We refer to \cite{livre1} for an account of the main 
results of convex analysis and monotone operator theory.

Let $(\Omega, \FF,\mathsf{P})$ be a probability space.
A $\HH$-valued random variable is a measurable (strong and weak measurability
coincide since $\HH$ is separable) function $X\colon \Omega\to\HH$, 
where $\HH$ is endowed with the Borel $\sigma$-algebra. We denote by 
$\sigma(X)$ the  $\sigma$-field generated by $X$.
The expectation of a random variable $X$ is denoted by $\E[X]$. The conditional expectation
of $X$ given a $\sigma$-field $\EuScript{A}\subset \EuScript{F}$ is denoted by 
$\E[X|\EuScript{A}]$. Given a random variable $Y\colon\Omega\to\HH$,
the conditional expectation of $X$ given $Y$ is denoted by $\E[X|Y]$. See \cite{LedTal91} for more
details on probability theory in Hilbert spaces.  A $\HH$-valued random process is a sequence $(x_n)_{n\in\NN}$ 
of $\HH$-valued random variables.
The abbreviation a.s. stands for ``almost surely''.

\begin{lemma}{\rm\cite[Theorem 1]{Rob85}}\label{l:rob85}
Let $(\FF_n)_{n\in\NN}$ be an increasing sequence of  sub-$\sigma$-algebras of ${\EuScript{F}}$,
let $(z_n)_{n\in\NN}$, $(\xi_n)_{n\in\NN}$, $(\zeta_n)_{n\in\NN}$ and $(t_n)_{n\in\NN}$ be 
$\RP$-valued random sequences such that,  for every $n\in\NN$, $z_n$, $\xi_n$, $\zeta_n$,
and $t_n$  are $\FF_n$-measurable. Assume moreover that  $\sum_{n\in\NN}t_n<+\infty$,
$\sum_{n\in\NN} \zeta_n<+\infty$ a.s.,  and 
\begin{equation}
(\forall n\in\NN)\quad \E[z_{n+1}|\FF_n] \leq (1+t_n)z_n + \zeta_n-\xi_n 
\quad \text{a.s.}.
\end{equation}
Then $(z_n)_{n\in\NN}$ converges a.s. and $(\xi_n)_{n\in\NN}$ is summable a.s..
\end{lemma}
The following lemma is a special case of \cite[Proposition 2.4]{Vu15}.
\begin{lemma}
\label{p:fejer}
Let $C$ be a non-empty closed subset of $\HH$,  
let $\alpha\in \left]0,\infty\right[$, let $W\in \BL(\HH)$ and $(W_n)_{n\in\NN}\subset \BL(\HH)$ be  
self-adjoint and strongly positive operators with constant $\alpha$,
such that $W_n\to W$ pointwise, and
let  $(x_n)_{n\in\NN}$ be a  $\HH$-valued random process. 
Suppose that, for every $x\in C$, there exist $\RP$-valued random 
sequences $(\xi_n(x))_{n\in\NN}$, $(\zeta_n(x))_{n\in\NN}$ and $(t_n(x))_{n\in\NN}$
such that,  for every $n\in\NN$, $\xi_n(x)$, $\zeta_n(x)$ and $t_n(x)$
are ${\EuScript{F}}_n$-measurable, $(\zeta_n(x))_{n\in\NN}$ and $(t_n(x))_{n\in\NN}$ are summable a.s.,
and\begin{equation}
 \label{eq:rob}
(\forall n\in\NN)\quad \E[\|x_{n+1}-x\|_{W_{n+1}}|\FF_n] \leq (1+t_n(x))\|x_{n}-x\|_{W_n}+ \zeta_n(x)-\xi_n(x)
\quad \quad \text{a.s.}
\end{equation}
Then the following hold.
\begin{enumerate}
\item 
\label{p:fejerii} $(x_n)_{n\in\NN}$ is bounded a.s. and $(\xi_n(x))_{n\in\NN}$ is summable a.s.
\item 
\label{p:fejeri}  There exists $\widetilde{\Omega}\subset \Omega$ such that $\PPP(\widetilde{\Omega})=1$ 
and, for every $\omega\in\widetilde{\Omega}$ and $x\in C$,
 $( \|x_{n}(\omega)-x\|)_{n\in\NN}$ converges.
\item 
\label{p:fejeriv}
Suppose that the set of weak cluster points of $(x_n)_{n\in\NN}$ is a subset of $C$ a.s. Then
$(x_n)_{n\in\NN}$ converges weakly a.s. to a $C$-valued random vector.
\end{enumerate}
\end{lemma}

\begin{lemma}{\rm \cite[Lemma 3.7]{optim2}}
\label{l:maxmon45}
Let $A\colon\HH \to 2^{\HH}$ be maximally monotone, 
 let $U\in\BL(\HH)$ be self-adjoint and strongly positive, and let 
$\GG$ be the real Hilbert space obtained by endowing $\HH$ with 
the scalar product 
$(x,y)\mapsto\scal{x}{y}_{U^{-1}}$.
Then, the following hold.
\begin{enumerate}
\item
\label{l:maxmon45i}
$UA\colon\GG \to 2^{\GG}$ is maximally monotone.
\item
\label{l:maxmon45ii}
$J_{UA}\colon\GG\to\GG$ is firmly nonexpansive.
\end{enumerate}
\end{lemma}
\section{Algorithm and almost sure convergence}
\label{sec:main}
In this section we state our main results.
We introduce the extension to monotone inclusions of 
the primal-dual algorithm in \cite{DroSabTeb15,LorVer11}, allowing for stochastic errors
in the evaluation of the operator $B$ and we prove
 almost sure weak convergence of the iterates.
\begin{algorithm}
\label{algo:1}
Let $(\gamma_n)_{n\in\NN}$ and $(\tau_n)_{n\in\NN}$ be sequences 
of strictly positive real numbers,
let $U$ be a self adjoint positive definite on $\GG$,
let $(r_{n})_{n\in\NN}$  be  a $\HH$-valued, square integrable random process,
let $x_{0}$ be  a $\HH$-valued, squared integrable random variable and 
let $v_{0}$ be  a $\GG$-valued, squared integrable random variable.
Iterate
\begin{equation}
\label{e:Tsengaa}
(\forall n\in\NN)\quad
\begin{array}{l}
\left\lfloor
\begin{array}{l}
p_{n} = P_V(x_{n} - \gamma_n(L^* v_{n} + r_{n}))\\
v_{n+1} = J_{\frac{\tau_n}{\gamma_n}UA^{-1}}\big(v_{n}+ \frac{\tau_n}{\gamma_n}U Lp_{n}\big)\\
x_{n+1} = P_V( x_{n} - \gamma_n(L^*v_{n+1} + r_{n})).\\
\end{array}
\right.\\[2mm]
\end{array}
\end{equation}
\end{algorithm}
{\bf Almost sure convergence.} 
We first establish almost sure convergence of the iterates generated by Algorithm \eqref{algo:1} under suitable conditions 
on the parameters $(\gamma_n)_{n\in\NN}$ and $(\tau_n)_{n\in\NN}$  as well as on the stochastic estimates of $B$.
\begin{theorem} 
\label{t:1}
In the setting of Problem~\ref{p:prob}, suppose that $\mathcal{P}$ is non empty and consider
algorithm~\ref{algo:1}.
Assume that the following conditions are satisfied for $\FF_n = \sigma((x_k,v_k)_{0\leq k\leq n})$
\begin{enumerate}
\item
\label{con:t:1i} $(\gamma_n)_{n\in\NN}$ is decreasing and $(\tau_n)_{n\in\NN}$ is increasing 
\item 
\label{con:t:1ii}
$(\exists \tau\in\left]0,+\infty\right[)$ such that $\sup_{n\in\NN} \tau_n\leq \tau$, and $(\tau U)^{-1}- LP_VL^*$ is positive definite.
\item 
\label{con:t:1iii}
$\gamma_0\in\left]0,\beta\right[$ and $\inf_{n\in\NN} \gamma_n >0$.
\item 
\label{con:t:1iv}$(\forall n\in\NN) \quad  \E[r_{n}|\FF_n] = Bx_n$.
\item
\label{con:t:1v} 
$\sum_{n\in \NN} \E[\|r_{n}-Bx_{n}\|^2 |\FF_n] < +\infty $  \quad \text{$\mathsf{P}$-a.s.}
\end{enumerate}
Then the following hold for some random vector $(x,v)$, $\mathcal{P}\times\mathcal{D}$-valued a.s.
\begin{enumerate}
\item $(x_n)_{n\in\NN}$ converges weakly to $x$ and $ (v_n)_{n\in\NN}$ converges weakly to $v$ a.s.
\item $\sum_{n\in\NN}\|Bx_n-Bx\|^2 < +\infty$ a.s.
\end{enumerate}
\end{theorem}
\begin{proof}
Let $\boldsymbol{\mathcal{S}}$ be the set of all points $(\overline{x},\overline{v})\in\HH\times\GG$ such that
\begin{equation}
\label{e:kt}
-L^*\overline{v}\in B\overline{x} + N_V\overline{x}
\quad
\text{and}
\quad
L\overline{x}\in A^{-1}\overline{v}.
\end{equation}
Then $\boldsymbol{\mathcal{S}}$ is non empty  and it is a closed convex subset of $\HH\times\GG$ \cite[Proposition 2.8]{siop2} contained in $\mathcal{P}\times\mathcal{D}$.
Since $A^{-1}$ is maximally monotone,  it follows from Lemma~\ref{l:maxmon45} that $(\tau_n/\gamma_n)UA^{-1}$ is maximally monotone with respect to 
$\scal{\cdot}{\cdot}_{(\gamma_n/\tau_n)U^{-1}}$. Hence, $J_{(\tau_n/\gamma_n)UA^{-1}}$ is firmly nonexpansive with respect to 
the norm $\| \cdot\|_{(\gamma_n/\tau_n)U^{-1}}$. Moreover, we also derive from \eqref{e:kt} that 
\begin{equation}
(\forall n\in\NN)\qquad \overline{v} = J_{\frac{\tau_n}{\gamma_n}UA^{-1}}\Big(\overline{v} +\frac{\tau_n}{\gamma_n} UL\overline{x}\Big).
\end{equation}
Therefore,
\begin{alignat}{2}
\label{est:1}
\gamma_n\| v_{n+1}-&\overline{v}\|_{(\tau_nU)^{-1}}^2 \notag\\
&\leq \gamma_n\|v_{n}-\overline{v} + \frac{\tau_n}{\gamma_n}UL(p_n-\overline{x})\|_{(\tau_nU)^{-1}}^2
- \gamma_n\| v_{n}-v_{n+1} + \frac{\tau_n}{\gamma_n}UL(p_n-\overline{x})\|_{(\tau_nU)^{-1}}^{2}.
\end{alignat}
We have 
\begin{alignat}{2}
\gamma_n \|v_{n}-&\overline{v} + \frac{\tau_n}{\gamma_n}UL(p_n-\overline{x})\|_{(\tau_nU)^{-1}}^2\notag\\ 
&=\gamma_n\|v_{n}-\overline{v} \|_{(\tau_nU)^{-1}}^2
+2\scal{L^*(v_{n}-\overline{v}) }{p_n-\overline{x}}+\frac{1}{\gamma_n} \|L(p_n-\overline{x})\|_{(\tau_nU)}^{2}
\end{alignat}
and 
\begin{alignat}{2}
\gamma_n\| v_{n}-&v_{n+1} + \frac{\tau_n}{\gamma_n} UL(p_n-\overline{x})\|_{(\tau_nU)^{-1}}^{2}\notag\\
&= \gamma_n\| v_{n+1}-v_n\|_{(\tau_nU)^{-1}}^2 + 2\scal{L^{*}(v_{n}-v_{n+1})}{p_n-\overline{x}} + \frac{1}{\gamma_n}\|L(p_n-\overline{x}) \|_{(\tau_nU)}^2.
\end{alignat}
By inserting the last two equalities into \eqref{est:1}, we obtain
 \begin{alignat}{2}
\label{e:esta1}
\gamma_n\| v_{n+1}-\overline{v}\|_{(\tau_nU)^{-1}}^2 \leq \gamma_n\|v_{n}-\overline{v} \|_{(\tau_nU)^{-1}}^2- \gamma_n\| v_{n+1}-v_n\|_{(\tau_nU)^{-1}}^2
+2\scal{L^*(v_{n+1}- \overline{v})}{p_n-\overline{x}}.
\end{alignat}
Let us now estimate the last term in \eqref{e:esta1}. Since $P_V$ is self-adjoint, $P_VP_V =P_V$, and $x_{n+1}-\overline{x}\in V$,
we have
\begin{alignat}{2}
\label{e:esta1a}
\scal{L^*(v_{n+1}- \overline{v})}{p_n-\overline{x}} 
&= \scal{L^*(v_{n+1}- \overline{v})}{x_{n+1}-\overline{x}} + \scal{L^*(v_{n+1}- \overline{v})}{p_n-x_{n+1}}\notag\\
&= \scal{L^*(v_{n+1}- \overline{v})}{P_V(x_{n+1}-\overline{x})} +\gamma_n \scal{L^*(v_{n+1}- \overline{v})}{P_VL^*(v_{n+1}- v_{n})}\notag\\
&= \scal{P_VL^*(v_{n+1}- \overline{v})}{x_{n+1}-\overline{x}} +\gamma_n \scal{P_VL^*(v_{n+1}- \overline{v})}{P_VL^*(v_{n+1}- v_{n})}.
\end{alignat}
Upon setting $M= LP_VL^*$, we have
\begin{alignat}{2}
\label{e:esta1b}
2\gamma_n \scal{(v_{n+1}- \overline{v})}{M(v_{n+1}- v_{n})} 
&= \gamma_n \|v_{n+1}- \overline{v}\|_M^2 + \gamma_n \| v_{n+1}-v_n \|_M^2-\gamma_n \|v_{n}- \overline{v}\|_M^2
\end{alignat}
Set
\begin{equation} 
(\forall n\in\NN)\;
\begin{cases}
 \chi_n&= 2\scal{P_VL^*(v_{n+1}- \overline{v})}{x_{n+1}-\overline{x}} ,\\
\zeta_n &= 2\scal{r_{n}-Bx_{n}}{x_{n+1}-\overline{x}},
\end{cases}
\end{equation}
and note that, by~\eqref{e:esta1},\eqref{e:esta1a}, and \eqref{e:esta1b}
\begin{equation}
\label{e:vn}
\gamma_n \|v_{n+1}- \overline{v}\|_{(\tau_nU)^{-1}-M}^2 \leq \gamma_n \|v_{n}- \overline{v}\|_{(\tau_nU)^{-1}-M}^2 -\gamma_n \|v_{n+1}-{v_n}\|_{(\tau_nU)^{-1}-M}^2 +\chi_n
\end{equation}
Equation \eqref{e:kt} yields 
\begin{align}
\label{e:addv}
\nonumber P_VL^*v_{n+1}&=\frac{1}{\gamma_n}\Big(P_Vx_n-x_{n+1}\Big)-P_Vr_n\\
-P_VL^*\overline{v}&=P_V B\overline{x}.
\end{align}
Using the monotonicity and the cocoercivity of $B$, and noting that  $x_{n+1}-\overline{x}\in V$, \eqref{e:addv} yields
\begin{alignat}{2}
\label{e:esta1c}
\chi_n+\zeta_n&= \frac{2}{\gamma_n}\scal{P_V x_{n}-x_{n+1}}{x_{n+1}-\overline{x}}-2\scal{P_V r_{n}-r_{n}}{x_{n+1}-\overline{x}}
-2\scal{Bx_{n}-P_VB\overline{x}}{x_{n+1}-\overline{x}}
\notag\\
&=\frac{2}{\gamma_n}\scal{x_{n}-x_{n+1}}{x_{n+1}-\overline{x}}
-2\scal{Bx_{n}- B\overline{x} }{x_{n}-\overline{x}}-2\scal{Bx_{n}- B\overline{x} }{x_{n+1}-{x_n}} \notag\\
&\leq \frac{2}{\gamma_n}\scal{x_{n}-x_{n+1}}{x_{n+1}-\overline{x}}
+2\scal{Bx_{n}- B\overline{x} }{x_{n}-x_{n+1}} - 2\beta \| Bx_{n}-B\overline{x}\|^2\notag\\
&\leq \frac{2}{\gamma_n}\scal{x_{n}-x_{n+1}}{x_{n+1}-\overline{x}} - 2\beta \| Bx_{n}-B\overline{x}\|^2
+\beta \| Bx_{n}-B\overline{x}\|^2  + \beta^{-1} \| x_{n}-x_{n+1}\|^2\notag\\
&\leq \frac{1}{\gamma_n}\big( \| x_{n}-\overline{x}\|^2 -\|x_{n+1}-\overline{x} \|^2  \big) 
- \Big(\frac{1}{\gamma_n}-\frac{1}{\beta}\Big) \| x_{n}-x_{n+1}\|^2- \beta \| Bx_{n}-B\overline{x}\|^2.
\end{alignat}
Set $R_n = \gamma_n^2\left((\tau_nU)^{-1}-M\right)$ and define 
\begin{equation}
\label{e:esta1d}
|||\cdot |||_n^{2}\colon \HH\times\GG \ni (x,v) \mapsto  \|x\|^2 + \|v\|^{2}_{R_n}
\quad \text{and}\quad \overline{\xx} = (\overline{x},\overline{v}), 
(\forall n\in\NN)\quad \xx_n = (x_n,v_n).
\end{equation}
We derive from \eqref{e:vn}, \eqref{e:esta1c}, and \eqref{e:esta1d} that 
\begin{alignat}{2}
\label{e:fejer1}
|||\xx_{n+1}-\overline{\xx} |||_n^{2} &\leq |||\xx_{n}-\overline{\xx} |||_n^{2}-\Big(1-\frac{\gamma_n}{\beta}\Big) \| x_{n+1}-x_{n}\|^2
- \| v_{n+1}-v_{n}\|_{R_n}^2 -\gamma_n\zeta_n- \gamma_n\beta \| Bx_{n}-B\overline{x}\|^2.
\end{alignat}
Let us set 
\begin{equation}
\label{e:exact}
\begin{cases}
\overline{p}_{n} =  P_V(x_{n} - \gamma_n(L^* {v}_{n} + Bx_{n}))\\
\overline{v}_{n+1} = J_{(\tau_n/\gamma_n)UA^{-1}}({v}_{n}+\frac{\tau_n}{\gamma_n} U L\overline{p}_{n})\\
\overline{x} _{n+1} =P_V( x_{n} - \gamma_n(L^* \overline{v}_{n+1} + Bx_{n}) ).\\
\end{cases}
\end{equation}
It follows from the nonexpansiveness of $J_{(\tau_n/\gamma_n)UA^{-1}}$ with respect to $\|\cdot\|_{(\gamma_n/\tau_n)U^{-1}}$ and the 
nonexpansiveness of $P_V$ that
\begin{alignat}{2}
\label{e:asin}
\|x_{n+1}-\overline{x}_{n+1}\|&\leq \gamma_n\|L^*(v_{n+1}-\overline{v}_{n+1}) +r_n-Bx_n\|\notag\\
&\leq \gamma_n \Big(\|L\|\|\overline{v}_{n+1}-v_{n+1}\|+\|r_n-Bx_n\| \Big)\notag\\
&\leq \gamma_n \Big(({\tau_n}/{\gamma_n})^{1/2}\|L\|^2 \|U\|^{1/2}\|\overline{p}_n-p_n\|+\|r_n-Bx_n\| \Big)\notag\\
&\leq  \gamma_n \Big((\tau_n\gamma_n)^{1/2} \|U\|^{1/2}\|L\|^2\|r_n-Bx_n\|+\|r_n-Bx_n\| \Big).\notag\\
&\leq  \gamma_n ((\tau_n\gamma_n)^{1/2}\|U\|^{1/2}\|L\|^2+1)\|r_n-Bx_n\|.
\end{alignat}
Hence, upon setting $\sigma_n = \gamma_n ((\tau_n\gamma_n)^{1/2}\|U\|^{1/2}\|L\|^2+1)$,
\begin{alignat}{2}
\label{e:asin1}
(\forall n\in\NN)\quad -\zeta_n 
&= -2\scal{r_{n} - Bx_{n}}{\overline{x}_{n+1}-\overline{x}}- 2\scal{r_{n} - Bx_{n}}{x_{n+1}-\overline{x}_{n+1}}\notag\\
&\leq -2\scal{r_{n} - Bx_{n}}{\overline{x}_{n+1}-\overline{x}} + 2\|r_{n} - Bx_{n} \| \| x_{n+1}-\overline{x}_{n+1}\|\notag\\
&\leq -2\scal{r_{n} - Bx_{n}}{\overline{x}_{n+1}-\overline{x}}+ \sigma_n\|r_{n} - Bx_{n} \|^2,
\end{alignat}
Since $(\gamma_n)_{n\in\NN}$ is decreasing and $(\tau_n)_{n\in\NN}$ is increasing, we derive from inequality \eqref{e:fejer1} that
\begin{alignat}{2}
\label{e:ee1s}
 |||\xx_{n+1}-\overline{\xx} |||_{n+1}^{2} 
&\leq |||\xx_{n}-\overline{\xx} |||_n^{2}-\Big(1-\frac{\gamma_n}{\beta}\Big) \| x_{n+1}-x_{n}\|^2
- \| v_{n+1}-v_{n}\|_{R_n}^2 - \gamma_n\beta \| Bx_{n}-B\overline{x}\|^2\notag\\
&\quad-2\gamma_n\scal{r_{n} - Bx_{n}}{\overline{x}_{n+1}-\overline{x}} + \gamma_n\sigma_n\|r_{n} - Bx_{n} \|^2 .
\end{alignat}

Since $L$, $B$, $J_{(\tau_n/\gamma_n)UA^{-1}}$, and $P_V$ are continuous, and  $x_n$ is $\FF_n$-measurable, $\overline{x}_{n+1}-\overline{x}$ is $\FF_n$-measurable, and by  \ref{con:iii}, we obtain 
\begin{alignat}{2}
(\forall n\in\NN)\quad \E[\scal{r_{n} - Bx_{n}}{\overline{x}_{n+1}-\overline{x}}|\FF_n]=\scal{ \E[r_{n} - Bx_{n}|\FF_n]}{\overline{x}_{n+1}-\overline{x}}=0.
\end{alignat}
Therefore, by taking conditional expectation with respect to $\FF_n$ of both sides of \eqref{e:ee1s}, we obtain
\begin{alignat}{2}\label{e:ppsa}
(\forall n\in\NN)\quad 
\E[|||\xx_{n+1}-\overline{\xx} |||_{n+1}^{2}  |\FF_n]  
&\leq |||\xx_{n}-\overline{\xx} |||_n^{2} - 
\E\bigg[ \Big(1-\frac{\gamma_n}{\beta}\Big) \| x_{n+1}-x_{n}\|^2
+\| v_{n+1}-v_{n}\|_{R_n}^2  \Big|\FF_n\bigg]\notag\\
&\quad\gamma_n\sigma_n \E[\|r_{n} - Bx_{n} \|^2 |\FF_n]-\gamma_n\beta \| Bx_{n}-B\overline{x}\|^2.
\end{alignat}
Since the sequence $(\gamma_n\sigma_n\E[\|r_{n} - Bx_{n} \|^2 |\FF_n])_{n\in\NN}$ 
is summable a.s by assumptions \ref{con:t:1ii},\ref{con:t:1iii}  and \ref{con:t:1v}, in view of  \eqref{eq:rob}, \eqref{e:ppsa} shows that 
$(\xx_n)_{n\in\NN}$  is a stochastic quasi-Fej\'{e}r monotone sequence with respect to
the target set $\boldsymbol{\mathcal{S}}$. Therefore, it follows from Lemma \ref{p:fejer} and condition~\ref{con:t:1iii} that 
\begin{equation}
\label{e:bas0}
\begin{cases}
(\xx_n)_{n\in\NN}\; \text{is bounded a.s.}\\ 
\sum_{n\in\NN} \E\big[\big(1-\frac{\gamma_n}{\beta}\big) \| x_{n+1}-x_{n}\|^2
+\| v_{n+1}-v_{n}\|_{R_n}^2  \big|\FF_n\big] < +\infty\; \text{ a.s.}\\
\sum_{n\in\NN} \| Bx_{n}-B\overline{x}\|^2 < +\infty \; \text{ a.s.}
\end{cases}
\end{equation} 
Since $L, B, P_V$ and $J_{(\tau_n/\gamma_n)UA^{-1}}$ are continuous, for every $n\in\NN$,  $\overline{x}_{n+1}$, 
$\overline{v}_{n+1}$ and $ \overline{p} _{n}$ are $\FF_n$-measurable. 
Therefore, for every $n\in\NN$,
\begin{alignat}{2}
\label{e:bas}
\|\overline{x}_{n+1}-x_{n}\|^2 + \|\overline{v}_{n+1}-v_{n}\|_{U^{-1}}^2  
&= \E[ \|\overline{x}_{n+1}-x_{n}\|^2 |\FF_n] + \E[ \| \overline{v}_{n+1}-v_{n}\|_{U^{-1}}^2 |\FF_n] \notag\\
&\leq 2\bigg(\E[ \|\overline{x}_{n+1}-x_{n+1}\|^2 |\FF_n] +  \E[ \|x_{n+1}-x_{n}\|^2 |\FF_n] \notag\\
&\quad+ \E[ \|\overline{v}_{n+1}-v_{n+1}\|_{U^{-1}}^2 |\FF_n] +  \E[ \|v_{n+1}-v_{n}\|_{U^{-1}}^2 |\FF_n] 
\bigg).
\end{alignat}
Note that, since $(\tau_n/\gamma_n)A^{-1}$ is maximally monotone,  $J_{U(\tau_n/\gamma_n)A^{-1}}$ is  
firmly nonexpansive with respect to $\scal{\cdot}{\cdot}_{U^{-1}}$ by Lemma~\ref{l:maxmon45}. Thus,
\begin{alignat}{2}
\label{e:bas1}
\E[\|\overline{v}_{n+1}-v_{n+1}\|_{U^{-1}}^2 |\FF_n] 
& \leq (\tau_n/\gamma_n) \E[ \| U L(\overline{p}_{n}-p_{n})\|_{U^{-1}}^2|\FF_n] \notag\\
&\leq \tau\gamma_0 \E[ \| U LP_V(r_n-Bx_{n})\|_{U^{-1}}^2|\FF_n] \to 0,
\end{alignat}
and hence
\begin{alignat}{2}
\label{e:bas2}
\E[ \|\overline{x}_{n+1}-x_{n+1}\|^2 |\FF_n] 
& \leq2 \gamma_0^2 \E[\| P_V L^*(\overline{v}_{n+1}-v_{n+1})\|^2|\FF_n] + 
2\gamma_0^2 \E[\| P_V(r_n- Bx_{n})\|^2|\FF_n]  \to 0.
\end{alignat}
Now using  \eqref{e:bas0}, \eqref{e:bas1}, and \eqref{e:bas2},  we derive from 
\eqref{e:bas} that 
\begin{equation}
\label{e:baba1}
\overline{x}_{n+1}-x_{n} \to 0
\quad \text{and}\quad 
\overline{v}_{n+1}-v_{n} \to 0,
\quad \text{and}\quad 
\overline{x}_{n+1}-\overline{p}_{n} \to 0
\quad \quad \text{a.s.}
\end{equation}
Moreover, it follows from the third line of \eqref{e:bas0} that 
\begin{equation}
\label{e:baba2}
Bx_n \to B\overline{x} \quad\text{a.s.},
\end{equation}
and from \eqref{e:exact} that 
\begin{equation}
\label{e:baba3}
U^{-1}(v_n-\overline{v}_{n+1}) \in \frac{\tau_n}{\gamma_n}\big(A^{-1}\overline{v}_{n+1} - L\overline{p}_n\big) 
\quad \text{and} \quad  \gamma_n^{-1}(x_n-\overline{x}_{n+1}) - Bx_n \in L^*\overline{v}_{n+1}+N_V\overline{x}_{n+1},
\end{equation}
almost surely.
Next, let us prove that every weak cluster point of 
$(\xx_n)_{n\in\NN}$ is in $\boldsymbol{\mathcal{S}}$ a.s. 
Let ${\Omega}_0$ be the set of all $\omega \in {\Omega}$ 
such that $(\xx_n(\omega))_{n\in\NN}$ is bounded 
and \eqref{e:baba1},\eqref{e:baba2} and \eqref{e:baba3} are satisfied. We have 
${\mathsf{P}}(\Omega_0) =1$. Fix $\omega \in \Omega_0$.
Let $\xx(\omega) = (x(\omega), v(\omega))$ be a weak cluster point of $(\xx_n(\omega))_{n\in\NN}$. 
Then there exists a subsequence $(\xx_{k_n}(\omega))_{n\in\NN}$ that converges weakly to
$\xx(\omega)$. 
By  \eqref{e:baba1}
\begin{equation}
\label{e:baba4}
\overline{x}_{k_n+1}(\omega)-x_{k_n}(\omega) \to 0
\quad \text{and}\quad 
\overline{v}_{k_n+1}(\omega)-v_{k_n}(\omega) \to 0,
\quad \text{and}\quad 
\overline{x}_{k_n+1}(\omega)-\overline{p}_{k_n}(\omega) \to 0
\end{equation}
and hence by \eqref{e:baba2}, and \cite[Proposition 20.33(ii)]{livre1}
\begin{equation} 
(\overline{x}_{k_n+1}(\omega), \overline{v}_{k_n+1}(\omega),\overline{p}_{k_n}(\omega))
\weakly  (x(\omega), v(\omega),x(\omega))
\quad 
\text{and}\quad 
Bx_{k_n}(\omega) \to B\overline{x} = Bx(\omega).
\end{equation} 
The operator $(x,p,v)\in\HH\times\HH\times\GG \mapsto (N_V(x),0,A^{-1}v) + (L^*v,0,-Lx)$ is maximally monotone 
by \cite[Corollary 24.4(i)]{livre1}, since it is the sum of two maximally monotone operators \cite[Proposition 20.23 and Example 20.30]{livre1}.
Hence, its graph is sequentially closed in $\HH^{\mathrm{weak}}\times\HH^{\mathrm{strong}}$ by \cite[Proposition 20.33(ii)]{livre1},
and we derive from \eqref{e:baba3} and \eqref{e:baba4} that 
\begin{equation}
-Bx(\omega) = L^*v(\omega) + N_Vx(\omega)
\quad \text{and}\quad 
Lx(\omega) \in A^{-1} v(\omega),
\end{equation}
which shows that $(x(\omega),v(\omega))\in \boldsymbol{\mathcal{S}}$ by \eqref{e:kt}. 
Altogether, it follows Lemma \ref{p:fejer}\ref{p:fejeriv}
that $\xx_n\weakly \xx$, with $\xx$ which is $\mathcal{P}\times\mathcal{D}$-valued a.s.
\end{proof}

A direct corollary of  the above theorem is the exact case,
where no stochastic errors occur.
\begin{corollary}
Suppose that $\mathcal{P}$ is non empty.
Let $(\gamma_n)_{n\in\NN}$ be a decreasing sequence  of strictly positive real numbers such that $\gamma_0 < \beta$, 
let $(\tau_n)_{n\in\NN}$ be  an increasing sequence  of strictly positive real numbers such that $\tau=\sup_{n\in\NN}\tau_n<+\infty$, 
let $U$ be a self adjoint positive definite on $\GG$ such that $(\tau U)^{-1}- LP_VL^*$ is positive definite,
let $x_{0}\in \HH$ and $v_{0}\in \GG$.
Iterate
\begin{equation}
\label{e:Tsengaadet}
(\forall n\in\NN)\quad
\begin{array}{l}
\left\lfloor
\begin{array}{l}
p_{n} =  P_V( x_{n} - \gamma_n(L^* v_{n} + Bx_n))\\
v_{n+1} = J_{(\tau_n/\gamma_n)UA^{-1}}(v_{n}+ \frac{\tau_n}{\gamma_n}U Lp_{n})\\
x_{n+1} = P_V( x_{n} - \gamma_n(L^*v_{n+1} + Bx_n)).\\
\end{array}
\right.\\[2mm]
\end{array}
\end{equation}
Then the following hold for some  $(x,v)\in \mathcal{P}\times\mathcal{D}$-valued.
\begin{enumerate}
\item $(x_n)_{n\in\NN}$ converges weakly $x$ and $ (v_n)_{n\in\NN}$ converges weakly to $v$.
\item $\sum_{n\in\NN}\|Bx_n-Bx\|^2 < +\infty$.
\end{enumerate}
\end{corollary}

In the following remark we comment on the features of the proposed algorithm and 
we discuss relations with existing work. 
\begin{remark}\ 
\begin{itemize} 
\item Specific instances of the algorithm considered here has been studied in the deterministic case 
in two papers independently: in \cite{DroSabTeb15} to solve a saddle point problem, and in 
\cite{LorVer11} to minimize a regularized
least squares problem. 
\item The proposed algorithm extends the algorithm in \cite{DroSabTeb15,LorVer11}
in several directions. First, we consider monotone inclusions instead of
saddle point problems, and, second, we allow for stochastic evaluations
of the single valued cocoercive operator. Moreover, the analysis encompasses a
 variable step-size, in contrast to the fixed one considered in \cite{DroSabTeb15,LorVer11}.
Note that our algorithm gives the possibility to treat linear constraints differently from
what has been proposed in \cite[Section 4.2]{DroSabTeb15}, thanks to the presence of the 
projection step. Concerning the results, it is worth noting
that the proof of weak convergence relies on different tools than the ones in \cite{DroSabTeb15,LorVer11},
which are specialized to the variational case.
\item As in the deterministic setting, each iteration of the algorithm consists of three steps. Note that, differently
from the extra-gradient methods \cite{Kor76,Nem04}, the third step
does not require an additional evaluation of the operator $B$. When $V=\HH$, and the step-sizes are
constant, the algorithm is a specific instance of the one considered in \cite{RosVilVu15}. Due to the
variable step-size, we need to use the concept of variable metric stochastic Fej\'er monotonicity \cite{optim2,Vu15}. 
\end{itemize}
\end{remark}

\section{Saddle point problems}
\label{sec:erg}
We next prove some results on ergodic convergence of the duality gap for the case of 
minimization or saddle point problems.

\begin{problem}
\label{pro:pd}
Let $h\colon \HH \to \RR$ be a convex differentiable function with a $\beta^{-1}$-Lipschitz continuous gradient,
for some $\beta \in \left] 0,+\infty\right[$, let $g\in\Gamma_0(\GG)$ and $L\in\mathcal{B}(\HH,\GG)$, 
let $V$ be a closed vector subspace of $\HH$.
The primal problem 
is to 
\begin{equation}
\label{primal1}
\underset{x\in V}{\text{minimize}}\; h(x) +g(Lx),
\end{equation}
and the dual problem is to 
\begin{equation}
\label{dual1}
\underset{v\in\GG}{\text{minimize}}\; (h+\iota_V)^*(-L^*v) +g^*(v),
\end{equation}
Denote by $\mathcal{P}_V$ and $\mathcal{D}_V$ the set of solutions to \eqref{primal1} and \eqref{dual1}, respectively.
\end{problem}
We consider the convex-concave saddle point formulation of the above problem by setting
\begin{align}
\label{e:saddle}
K\colon \HH\times\GG&\to \mathbb{R}\cup\{-\infty,+\infty\}\notag\\
(x,v) &\mapsto h(x)+\iota_V(x)+ \scal{Lx}{v} -g^*(v).
\end{align}
We are interested in finding $(\overline{x},\overline{v})\in \mathcal{P}_V\times \mathcal{D}_V$, 
or equivalently (under suitable qualification conditions, see \cite[Theorem 15.23 and Proposition 19.18]{livre1}), a saddle point of \eqref{e:saddle}. 
We will consider the following notion of approximated saddle points, extending to the
stochastic case the one given in
\cite{NemYud83}. Let $\varepsilon>0$. A $\HH\times\GG$ valued random variable is an $\varepsilon$-
saddle point of $K$ in expectation if
\begin{equation}
\label{e:esaddle}
\sup_{(\overline{x},\overline{v})\in\mathcal{P}_V\times\mathcal{D}_V}\E\left[ K(z,\overline{v})-K(\overline{x},u)\right]\leq \epsilon
\end{equation}
\begin{algorithm}
\label{a:apd}
Let $(\gamma_n)_{n\in\NN}$ be a decreasing sequence of strictly positive real numbers, 
let $(\tau_n)_{n\in\NN}$ be an increasing sequence of strictly positive real numbers, 
let $U$ be a self adjoint positive definite linear operator on $\GG$,
let $(r_{n})_{n\in\NN}$  be  a $\HH$-valued, square integrable random process,
let $x_{0}$ be  a $V$-valued, squared integrable random variable and 
let $v_{0}$ be  a $\GG$-valued, squared integrable random variable.
Iterate
\begin{equation}
\label{e:Tsengaaa}
(\forall N\in\NN)\quad
\begin{array}{l}
\left\lfloor
\begin{array}{l}
\operatorname{For}\;n=0,\ldots,N\\
\left\lfloor
\begin{array}{l}
p_{n} = P_V(x_{n} - \gamma_n(L^* v_{n} + r_{n}))\\
v_{n+1} = \prox^{U^{-1}}_{(\!\tau_n/\gamma_n\!) g^{*}}(v_{n}+ \frac{\tau_n}{\gamma_n}U Lp_{n})\\
x_{n+1} =P_V( x_{n} - \gamma_n(L^*v_{n+1} + r_{n}))\\
\end{array}
\right.\\[2mm]
\tilde{x}_{N} =  \Big(\sum_{n=0}^N\gamma_n\Big)^{-1}\sum_{n=0}^N \gamma_n x_{n+1}\\
\tilde{v}_{N} =  \Big(\sum_{n=0}^N\gamma_n\Big)^{-1}\sum_{n=0}^N \gamma_n v_{n+1}.\\
\end{array}
\right.\\[2mm]
\end{array}
\end{equation}
\end{algorithm}
\begin{theorem}
\label{thm:erg}
In the setting of Problem \ref{pro:pd}, suppose that $\mathcal{P}_V$ is non empty. 
Moreover, assume that the following conditions are satisfied for 
Algorithm~\ref{a:apd}, with $\FF_n = \sigma((x_k,v_k)_{0\leq k\leq n})$
\begin{enumerate}
\item 
\label{con:ii} $\gamma_{0}\in\left]0,\beta\right[$. 
\item
\label{con:sigma} there exists $\tau\in\left]0,+\infty\right[$ such that  $\sup_{n\in\NN}\tau_n\leq \tau$ and  $(\tau U)^{-1}-LP_VL^*$ is positive semidefinite.
\item  $(\forall n\in\NN)\; \E[r_n |\FF_n] = \nabla h(x_n)$.
\item $c_0=\sum_{n\in\NN}\gamma_{n}^2 \E[\|r_{n}-\nabla h(x_{n})\|^2] < \infty$.
\end{enumerate}
Let $x\in V$ and let $v\in\dom g^*$.
Set 
\begin{equation*}
c(x,v)= \E[\| x_0-x\|^2 +\gamma_0^2\|v_0-v \|_{(\tau_0U)^{-1}-LP_VL^*}^2]
+ 2((\tau\gamma_0\|U\|)^{1/2} \|L\|^2+1)c_0.
\end{equation*}
Then 
\begin{equation}
\label{e:res1}
\E[K(\tilde{x}_N,v) - K(x,\tilde{v}_N) ]\leq \frac{c(x,v)}{2}\Big(\sum_{n=0}^N\gamma_n\Big)^{-1}.
\end{equation}
\end{theorem}
\begin{proof}
The Lipschitz continuity of $\nabla h$ and the convexity of $h$ imply that
\begin{equation}
\label{e:3p}
(\forall u\in\dom g^*) (\forall (t,y,z)\in V^3)\quad K(t,u) \leq K(y,u) + \scal{\nabla_xK(z,u)}{t-y} + \frac{1}{2\beta}\|t-z\|^2.
\end{equation}
Inequality \eqref{e:3p}, with $u= v_{n+1}$, $t= x_{n+1}$,  $y=x$ and $z=x_n$, yields
\begin{alignat}{2}
K(x_{n+1},v_{n+1}) 
&\leq K(x,v_{n+1}) + \scal{\nabla_xK(x_n,v_{n+1})}{x_{n+1}-x} + \frac{1}{2\beta}\|x_{n+1}-x_n\|^2\notag\\
&= K(x,v_{n+1}) + \scal{\nabla h(x_n) + L^*v_{n+1}}{x_{n+1}-x} + \frac{1}{2\beta}\|x_{n+1}-x_n\|^2\notag\\
&= K(x,v_{n+1}) + \scal{P_V( r_n + L^*v_{n+1})}{x_{n+1}-x} + \frac{1}{2\beta}\|x_{n+1}-x_n\|^2 -\zeta_n(x),
\end{alignat} 
where, for every $n\in\NN$, $\zeta_n(x)=\scal{P_V( r_n + L^*v_{n+1})}{x_{n+1}-x} -\scal{\nabla h(x_n) + L^*v_{n+1}}{x_{n+1}-x}$.
Let us set 
\begin{equation}
\label{e:exactm}
\begin{cases}
\overline{p}_{n} =  P_V(x_{n} - \gamma_n(L^* {v}_{n} + \nabla h(x_{n})))\\
\overline{v}_{n+1} =\prox^{U^{-1}}_{(\!\tau_n/\gamma_n\!) g^*}({v}_{n}+ \frac{\tau_n}{\gamma_n}U L\overline{p}_{n})\\
\overline{x} _{n+1} =P_V( x_{n} - \gamma_n(L^* \overline{v}_{n+1} + \nabla h(x_{n})) ).\\
\end{cases}
\end{equation}
Since $x_{n+1}-x \in V$, proceeding as in \eqref{e:asin}-\eqref{e:asin1}, we get
\begin{alignat}{2}
(\forall n\in\NN)\quad  -\zeta_n(x)
&= -\scal{P_V(r_{n} - \nabla h(x_{n}))}{\overline{x}_{n+1}-{x}}
- \scal{P_V( r_{n} - \nabla h(x_{n}))}{ x_{n+1}-\overline{x}_{n+1}}\notag\\
&\leq -\scal{r_{n} - \nabla h(x_{n})}{\overline{x}_{n+1}-{x}} + \sigma_n\|r_{n} - \nabla h(x_{n}) \|^2, 
\end{alignat}
where $\sigma_n = \gamma_n ((\tau_n\gamma_n\|U\|)^{1/2}\|L\|^2+1)$.
Using \eqref{e:Tsengaaa}, we get $P_V( r_n + L^*v_{n+1}) = \gamma^{-1}_n(x_n-x_{n+1})$, and hence 
\begin{alignat}{2}
\label{eq:222}
K(x_{n+1},&v_{n+1}) 
-K(x,v_{n+1}) \leq \gamma^{-1}_n\scal{x_n-x_{n+1} }{x_{n+1}-x} + \frac{1}{2\beta}\|x_{n+1}-x_n\|^2 -\zeta_n(x)\notag\\
&\leq \frac{1}{2\gamma_n}\big( \| x_n-x\|^2 - \|x_{n+1}-x\|^2 \big) 
-  \frac{1}{2}\Big( \frac{1}{\gamma_n}-\frac{1}{\beta}\Big)\|x_{n+1}-x_n\|^2\notag\\
&-\scal{r_{n} - \nabla h(x_{n})}{\overline{x}_{n+1}-{x}} + \sigma_n\|r_{n} - \nabla h(x_{n}) \|^2\notag\\
&\leq  \frac{1}{2\gamma_n} \Big(\| x_n-x\|^2 - \|x_{n+1}-x\|^2 \Big)
-\scal{r_{n} - \nabla h(x_{n})}{\overline{x}_{n+1}-x} + \sigma_n\|r_{n} - \nabla h(x_{n})\|^2.
\end{alignat} 
Therefore,
\begin{multline}
\gamma_n\left(K(x_{n+1},v_{n+1}) 
- K(x,v_{n+1})\right) \leq \frac{1}{2} \left(\| x_n-x\|^2 - \|x_{n+1}-x\|^2 \right)\\
-\gamma_n\scal{r_{n} - \nabla h(x_{n})}{\overline{x}_{n+1}-x} + \sigma_n\gamma_n\|r_{n} - \nabla h(x_{n}) \|^2
\end{multline}
Since
\begin{equation}
v_{n+1} = \prox_{(\!\tau_n/\gamma_n\!)g^*}^{U^{-1}}\Big(v_n+\frac{\tau_n}{\gamma_n}ULp_n\Big) = \prox^{U^{-1}}_{-(\!\tau_n/\gamma_n\!)K(p_n,\cdot)}(v_n),
\end{equation}
inequality \eqref{e:up} yields
\begin{equation}
\label{e:inp}
K(p_n,v)-K(p_n,v_{n+1}) \leq \frac{\gamma_n}{\tau_n}\scal{v-v_{n+1}}{U^{-1}(v_{n+1}-v_n)}.
\end{equation}
Now, simple calculation shows that 
\begin{equation}
K(x_{n+1},v)-K(x_{n+1},v_{n+1}) + K(p_n,v_{n+1})-K(p_n,v) = \scal{x_{n+1}-p_n}{L^*(v-v_{n+1})}.
\end{equation}
Therefore, using \eqref{e:inp} and setting, for every $n\in\NN$, $R_n=(\tau_nU)^{-1}-LP_VL^*$ we obtain
\begin{alignat}{2}
\label{eq:111}
\gamma_nK(x_{n+1},v)-&\gamma_nK(x_{n+1},v_{n+1}) = 
\gamma_n\scal{x_{n+1}-p_n}{L^*(v-v_{n+1})}+ \gamma_n(K(p_n,v)- K(p_n,v_{n+1}))\notag\\
&=  \gamma^{2}_n \scal{v_n-v_{n+1}}{LP_VL^*(v-v_{n+1})}+ \gamma_n(K(p_n,v)- K(p_n,v_{n+1}))\notag\\
&\leq\gamma^{2}_n \scal{v_n-v_{n+1}}{LP_VL^*(v-v_{n+1})}+  \frac{\gamma_n^2}{\tau_n}\scal{v-v_{n+1}}{U^{-1}(v_{n+1}-v_n)}\notag\\
&=\gamma^{2}_n \scal{v_{n+1}-v_{n}}{\big((\tau_nU)^{-1}-LP_VL^*\big)(v-v_{n+1})}\notag\\
&=\gamma_n^2 \Big(\frac12\|v_n-v \|_{R_n}^2-\frac12 \|v_{n+1}-v \|_{R_n}^2-\frac12 \|v_{n+1}-v_n \|_{R_n}^2\Big)\notag\\
\end{alignat}
Now, by adding \eqref{eq:111} and \eqref{eq:222}, and taking into account that $(\gamma_n)_{n\in\NN}$ is decreasing
and $(\tau_n)_{n\in\NN}$ is increasing, we get
\begin{alignat}{2}
\gamma_n\big(K(x_{n+1},v)\!-\!K(x,v_{n+1}) \big)
&\leq \frac{1}{2}\big( \| x_n-x\|^2\! -\! \|x_{n+1}-x\|^2 \big) 
+\frac{\gamma^2_n}{2}\|v_n-v \|_{R_n}^2\!-\!\frac{\gamma^2_{n+1}}{2} \|v_{n+1} - v \|_{R_{n+1}}^2 \notag\\
&\quad-\gamma_n\scal{r_{n} - \nabla h(x_{n})}{\overline{x}_{n+1}-x} + \gamma_n\sigma_n\|r_{n} - \nabla h(x_{n}) \|^2.
\end{alignat}
Since $K(\cdot,\cdot)$ is convex-concave, we have 
\begin{alignat}{2}
\label{basi}
K(\tilde{x}_N,v)-K(x,\tilde{v}_N) &\leq \Big(\sum_{n=0}^N\gamma_n\Big)^{-1} \sum_{n=0}^N \gamma_n(K(x_{n+1},v)-K(x,v_{n+1}) )\notag\\
&\leq \frac{1}{2}\Big(\sum_{n=0}^N\gamma_n\Big)^{-1}\Big(\| x_0-x\|^2 +\gamma^2_0\|v_0-v \|_{G_0}^2\notag\\
&\quad+ \sum_{n=0}^N\big(\gamma_n\sigma_n\|r_{n} - \nabla h(x_{n}) \|^2- \gamma_n\scal{r_{n} - \nabla h(x_{n})}{\overline{x}_{n+1}-x}\big)\Big).
\end{alignat}
Since for every $n\in\NN$, $\overline{x}_{n+1}$ is $\FF_n$-measurable, we have
\begin{alignat}{2}
(\forall n\in\NN)\quad \E[\scal{r_{n} - \nabla h(x_{n})}{\overline{x}_{n+1}-x}] 
&= \E[\E[\scal{r_{n} - \nabla h(x_{n})}{\overline{x}_{n+1}-x}|\FF_n]]\notag\\
&= \E[\scal{\E[r_{n} - \nabla h(x_{n})|\FF_n]}{\overline{x}_{n+1}-x}]\notag\\
&=0.
\end{alignat}
Therefore, by taking the expectation of both sides of \eqref{basi}, we obtain
\begin{alignat}{2}
\E[K(\tilde{x}_N,v)&-K(x,\tilde{v}_N) ] \notag \\
&\leq  \frac{1}{2}\Big(\sum_{n=0}^N\gamma_n\Big)^{-1}\Big(\E[\| x_0-x\|^2 +\gamma^2_0\|v_0-v \|_{G_0}^2]+ \sum_{n=0}^N\big(\gamma_n\sigma_n\E[\|r_{n} - \nabla h(x_n) \|^2])\Big)\notag\\
&\leq \frac{1}{2}\Big(\sum_{n=0}^N\gamma_n\Big)^{-1}\Big(\E[\| x_0-x\|^2 +\gamma^2_0\|v_0-v \|_{G_0}^2]+ 2((\gamma_0\tau\|U\|)^{1/2}\|L\|^2+1) c_0\Big),
\end{alignat} 
which proves \eqref{e:res1}. 
\end{proof}

\begin{corollary} 
Under the assumptions of Theorem~\ref{thm:erg}, suppose that $\sum_{n\in\NN}\gamma_n=+\infty$,
and that  $\mathcal{P}_V$ and $\mathcal{D}_V$ are bounded.
Let $\alpha=\sup_{(\overline{x},\overline{v})\in \mathcal{P}_V\times\mathcal{D}_V} \E[\| x_0-\overline{x}\|^2 +\gamma^2_0\|v_0-\overline{v} \|_{G_0}^2] + 2((\gamma_0\tau\|U\|)^{1/2}\|L\|^2+1) c_0$. Let $\varepsilon >0$, and let $N_\varepsilon\in\NN$ be such that $\sum_{n=1}^{N_\varepsilon}\gamma_n\leq 2/(\alpha\varepsilon)$.
Then $(\tilde{x}_{N_\varepsilon},\tilde{v}_{N_\varepsilon})$ is an $\varepsilon$-saddle point in expectation. 
\end{corollary}
\begin{proof}
The conclusion directly follows from \eqref{e:res1}, noting that $\sup_{(\overline{x},\overline{v})\in \mathcal{P}_V\times\mathcal{D}_V}c(\overline{x},\overline{v})\leq \alpha$.
\end{proof}

\section{Application to sum of composite inclusions}
\label{sec:appl}
Based on the standard product space reformulation technique, 
one can apply the proposed framework to more general composite inclusions. As an illustration,
we present below an application to the sum of composite operators and a cocoercive operator \cite{optim2,aicm1,jota1}.
\begin{problem}
\label{prob1} 
Let $\HH$ be a real Hilbert space,
let $m$ be a strictly positive integer and let $(\omega_i)_{1\leq i\leq m}
\in\left[0,1\right]^m$ be such that $\sum_{i=1}^m\omega_i=1$,
 let $C\colon\HH\to\HH$ 
be $\mu$-cocoercive for some 
$\mu\in\left]0,+\infty\right[$.
For every $i\in\{1,\ldots, m\}$, let $\GG_i$ be a real Hilbert 
space,  let $A_i\colon \GG_i\to2^{\GG_i}$ be maximally monotone, 
and suppose that $L_i\colon\HH \to\GG_i$ 
is a nonzero bounded linear operator. 
The problem is to solve the primal inclusion
\begin{equation}\label{e:primal1}
\text{find $\overline{x}\in\HH$ such that}\; 
0\in \sum_{i=1}^m\omega_iL^{*}_i
A_i(L_i\overline{x})+C\overline{x},
\end{equation}
together with the dual inclusion
\begin{alignat}{2}\label{e:dual1}
&\text{find $\overline{v}_1 \in \GG_1,\ldots, \overline{v}_m \in \GG_m$ 
such that }\notag \\
&\hspace{2cm}
(\exists x\in\HH)\quad-Cx=\sum_{i=1}^m\omega_i L_{i}^*\overline{v}_i \quad \text{and}\quad
(\forall i\in\{1,\ldots,m\})\quad \overline{v}_i\in
 A_{i}(L_ix).
\end{alignat}
We denote by $\mathcal{P}_1$ and $\mathcal{D}_1$ the sets of solutions 
to \eqref{e:primal1} and~\eqref{e:dual1}, respectively.
\end{problem}
Let us recall the following facts that show that  Problem \ref{prob1} is a special case of  
Problem \ref{p:prob} (see also \cite[Theorem 3.8]{siop2} for the case where $C$ is zero).
\begin{lemma} 
\label{l:lem1}
Define $\HHH=\HH^m$ and $\GGG = \GG_1\oplus\cdots\oplus \GG_m$, endowed with the scalar product
\begin{equation}
\label{e:scal}
(\forall (\vv,\ww)\in\GGG^2)\quad\scal{\vv}{\ww}_{\GGG}=\sum_{i=1}^m \omega_i \scal{v_i}{w_i}
\end{equation}
 and 
\begin{equation}
\begin{cases}
\VV =\menge{\xx\in\HHH}{x_1=\ldots=x_m}\notag\\
\AAA =  {\cart}_{k=1}^mA_i\\
\CCC \colon \HHH\to\HHH\colon \xx \mapsto (Cx_1,\ldots, Cx_m)\\
\LL \colon \HHH\to\GGG\colon \xx\mapsto (L_1x_1,\ldots,L_mx_m).
\end{cases}
\end{equation}
Then the following hold.
\begin{enumerate} 
\item \label{ab1}
$\CCC$ is $\mu$-cocoercive.
\item  \label{ab2}
$x$ solves \eqref{e:primal1} if and only if $(x,\ldots,x)\in \zer( \CCC + \LL^*\AAA\LL + N_{\VV})$. 
\end{enumerate}
\end{lemma}
\begin{proof}
\ref{ab1}: Since $C$ is $\mu$-cocoercive, we have
\begin{alignat}{2}
(\forall \xx\in\HH^m)(\forall \yy\in\HH^m)\quad 
\scal{\CCC\xx-\CCC\yy}{\xx-\yy} &= \sum_{i=1}^m\omega_i\scal{Cx_i-Cy_i}{x_i-y_i}\notag\\
&\geq \mu \sum_{i=1}^m\omega_i\|Cx_i-Cy_i\|^2 \notag\\
&= \mu \|\CCC\xx-\CCC\yy\|^2,
\end{alignat}
which shows that $\CCC$ is $\mu$-cocoercive on $\HHH$.

\ref{ab2}:
We have $\LL^*\colon (v_1,\ldots,v_m) \mapsto (L^{*}_1v_1,\ldots, L^{*}_mv_m)$. 
Moreover, in view of \eqref{e:scal}, $\VV^\perp=\menge{\xx\in\HHH}{\sum_{i=1}^m \omega_ix_i=0}$.
We have
\begin{alignat}{2}
x\; \text{solves \eqref{e:primal1}}\;&\Leftrightarrow 0\in \sum_{i=1}^m\omega_i L^{*}_i
A_i(L_i{x})+C{x}\notag\\
&\Leftrightarrow (\exists \big(y_i)_{1\leq i\leq m}\in \cart_{i=1}^m L^{*}_iA_iL_ix\big)
\quad 0 = \sum_{i=1}^m \omega_i(y_i) + Cx)\notag\\
&\Leftrightarrow (\exists \yy\in \cart_{i=1}^m L^{*}_iA_iL_ix\big)\quad 
-\yy - (Cx,\ldots,Cx) \in \VV^{\perp} = N_{\VV}(x,\ldots,x)\notag\\
&\Leftrightarrow (\exists \yy\in \cart_{i=1}^m L^{*}_iA_iL_ix\big)\quad 
0\in  \yy + (Cx,\ldots,Cx) + N_{\VV}(x,\ldots,x) \notag\\
&\Leftrightarrow 0\in  \LL^*\AAA\LL(x,\ldots,x)  + \CCC(x,\ldots,x) + N_{\VV}(x,\ldots,x) \notag\\
\end{alignat}
which proves \ref{ab2}.
\end{proof}
\begin{corollary}
\label{cor:comp}
 Let $(\gamma_n)_{n\in\NN}$ be a decreasing sequence of strictly positive real numbers
 and let $(\tau_n)_{n\in\NN}$ be an increasing sequence of strictly positive real numbers 
 such that $\sup_{n\in\NN} \tau_n=\tau<+\infty$
For every $i\in\{1,\ldots,m\}$  let $U_i$ be a self adjoint positive definite linear operator
 on $\GG_i$, let $(r_{n})_{n\in\NN}$  be  a $\HH$-valued, squared integrable random process,
 let $x_{0}$ be  a $\HH$-valued, squared integrable random variable and 
let $v_{i,0}$ be  a $\GG_i$-valued, squared integrable random variable.
 Iterate
 \begin{equation}
 \label{e:Tsengaaa1}
 (\forall n\in\NN)\quad
\begin{array}{l}
 \left\lfloor
 \begin{array}{l}
 p_{n} = x_n-\gamma_n \sum_{i=1}^m \omega_i( L^{*}_iv_{i,n} + r_{n})\\
\operatorname{For}\;i=1,\ldots,m\\
\left\lfloor
\begin{array}{l}
 v_{i,n+1} = J_{ (\tau_n/\gamma_n)U_iA^{-1}_i}(v_{i,n}+ \frac{\tau_n}{\gamma_n}U_i L_ip_{n})\\
\end{array}
\right.\\[2mm]
 x_{n+1} = x_{n} - \gamma_n\sum_{i=1}^m\omega_i(L^{*}_iv_{i,n+1} + r_{n})\\
 \end{array}
 \right.\\[2mm]
 \end{array}
 \end{equation}
  Suppose that $\mathcal{P}_1$ is non empty and 
 that the following conditions are satisfied for \\
$\FF_n = \sigma((x_{k},(v_{i,k})_{1\leq i\leq m})_{0\leq k\leq n})$
 \begin{enumerate}
 \item\label{con:i}$(\forall i\in \{1,\ldots,m\})\quad $ $(\tau U_i)^{-1}- L_iL_{i}^*$ is positive definite.
 \item $\gamma\in\left]0,\mu\right[$.
 \item \label{con:iii}$ \E[r_{n}|\FF_n] = Cx_{n}$.
 \item\label{con:iv} $\sum_{n\in\NN} \E[\|r_{n}-Cx_{n}\|^2 |\FF_n] < +\infty $  \quad \text{$\mathsf{P}$-a.s.}
 \end{enumerate}
 Then the following hold for some random vectors $\overline{x}\in\mathcal{P}_1$ a.s. and $(\overline{v}_1,\ldots,\overline{v}_m)\in\mathcal{D}_1$ a.s.
 \begin{enumerate}
\item $(x_{n})_{n\in\NN}$ converges weakly to $\overline{x}$ almost surely. 
\item For every $i\in \{1,\ldots,m\}$, $(v_{i,n})_{n\in\NN}$ converges weakly to $\overline{v}_i$ a.s.
 \item $\sum_{n\in\NN}\|Cx_n-Cx\|^2 < +\infty$ a.s.
\end{enumerate}
\end{corollary}
\begin{proof} 
Let us define 
\begin{equation}
\UU\colon\GGG\to \GGG\colon (v_1,\ldots,v_m) \mapsto (U_1v_1,\ldots,U_mv_m).
\end{equation}
Then $\UU$ is self adjoint, positive definite on $\GGG$. Moreover, for every $n\in\NN$,
\begin{equation}
(\forall \vv =(v_1,\ldots,v_m) \in\GGG)\quad J_{(\tau_n/\gamma_n)\UU\AAA^{-1}}\vv = (J_{(\tau_n/\gamma_n)U_iA^{-1}_i}v_i)_{1\leq i\leq m}.
\end{equation}
  We recall that 
\begin{equation}
(\forall \xx\in \HHH)\quad P_{\VV}\xx = \Big(\sum_{i=1}^m\omega_ix_i\Big)_{1\leq k\leq m}.
\end{equation}
Therefore, upon setting
\begin{equation}
\begin{cases}
\rr_n=(r_n,\ldots,r_n)\in \VV\\
\xx_n= (x_n,\ldots,x_n)\in\VV\\
\vv_n = (v_{1,n},\ldots,v_{m,n})\\
\pp_n = (p_n,\ldots,p_n)\in\VV,
\end{cases}
\end{equation}
we can rewrite \eqref{e:Tsengaaa1} as 
\begin{equation}
 \label{e:Tsengaaa2}
 (\forall n\in\NN)\quad
\begin{array}{l}
 \left\lfloor
 \begin{array}{l}
 \pp_{n} = P_{\VV}( \xx_n-\gamma_n ( \LL^{*}\vv_{n} + \rr_{n}))\\
 \vv_{n+1} = J_{ (\tau_n/\gamma_n)\UU\AAA^{-1}}(\vv_{n}+\frac{\tau_n}{\gamma_n} \UU \LL\pp_{n})\\
 \xx_{n+1} = P_{\VV}( \xx_{n} - \gamma_n(\LL^{*}\vv_{n+1} + \rr_{n})),\\
 \end{array}
 \right.\\[2mm]
 \end{array}
 \end{equation}
 which is a special case of \eqref{e:Tsengaa}. By Lemma \ref{l:lem1},
Problem \ref{prob1} is a special case of  Problem \ref{p:prob} Moreover, every specific condition in Theorem \ref{t:1} is satisfied. 
Therefore, the first and the third conclusions follow from Theorem \ref{t:1}. We prove the second one.  By Theorem \ref{t:1}, 
$(\vv_n)_{n\in\NN}$ converge weakly to $\vv$ a.s., such that 
\begin{equation}
\label{e:diem}
0\in -\LL(N_{\VV}+\CCC)^{-1}(-\LL^*\vv) +\AAA^{-1}\vv.
\end{equation}
We now prove that $\vv\in\mathcal{D}_1$. We have 
\begin{alignat}{2}
\eqref{e:diem}& \Rightarrow (\exists \xx\in\VV)\quad -\LL^*\vv \in N_{\VV}\xx +\CCC\xx \quad 
\text{and}\quad \vv\in\AAA\LL\xx.\notag\\
 &\Leftrightarrow (\exists \xx\in\VV)\quad -\LL^*\vv -\CCC\xx\in\VV^{\perp} \quad 
\text{and}\quad (\forall i\in\{1,\ldots, m\})\; v_i\in A_i L_ix.\notag\\
 &\Leftrightarrow (\exists x\in \HH)\quad -\sum_{i=1}^m\omega_iL_iv_i = Cx \quad 
\text{and}\quad (\forall i\in\{1,\ldots, m\})\; v_i\in A_i L_ix.\notag\\
&\Leftrightarrow \vv = (v_i)_{1\leq i\leq m}\in \mathcal{D}_1.
\end{alignat}
\end{proof}

\subsection{Structured minimization problems}
We next show how the previous result can be specialized to the case of minimization problems
involving the sum of composite functions \cite{ComPes12,optim2}. 
\begin{problem}
\label{prob2} 
Let $\HH$ be a real Hilbert space,
let $m$ be a strictly positive integer and let $(\omega_i)_{1\leq i\leq m}
\in\left[0,1\right]^m$ be such that $\sum_{i=1}^m\omega_i=1$,
 let $h\in\Gamma_0(\HH)$ be differentiable, with
$1/\mu$-Lipschitz continuous gradient for some 
$\mu\in\left]0,+\infty\right[$.
For every $i\in\{1,\ldots, m\}$, let $\GG_i$ be a real Hilbert 
space,  let $g_i\in\Gamma_0(\GG_i)$ be maximally monotone, 
and suppose that $L_i\colon\HH \to\GG_i$ 
is a nonzero bounded linear operator. 
The problem is to 
\begin{equation}\label{e:primalmin}
\minimize{x\in\HH}{\sum_{i=1}^m \omega_i g_i(L_i x) +h(x)}
\end{equation}
together with the dual
\begin{equation}
\label{e:dualmin}
\minimize{{v}_1 \in \GG_1,\ldots, {v}_m \in \GG_m} {\sum_{i=1}^m \omega_i g_i^*(v_i)+h^*\Big(-\sum_{i=1}^m\omega_i L^*_i v_i\Big).}
\end{equation}
under the assumption that problem~\eqref{e:primalmin} has at least a solution. 
\end{problem}
As in Section~\ref{sec:erg} we will also consider the saddle point formulation of Problem~\ref{prob2}.
\begin{equation}
H(x,v_1,\ldots,v_m)= h(x)+\sum_{i=1}^m \omega_i \scal{L^*_iv_i}{x} -\sum_{i=1}^m \omega_i g^*_i(v_i)
\end{equation}
The following algorithm is a special case of  Algorithm \ref{algo:1}.
\begin{algorithm}
\label{algo:min}
Let $(\gamma_n)_{n\in\NN}$ be a decreasing sequence of strictly positive real numbers, 
let $(\tau_n)_{n\in\NN}$ be an increasing sequence of strictly positive real numbers,
let $U$ be a self adjoint positive definite linear operator on $\GG$,
let $(r_{n})_{n\in\NN}$  be  a $\HH$-valued, square integrable random process,
let $x_{0}$ be  a $\HH$-valued, squared integrable random variable and 
let $v_{0}$ be  a $\GG$-valued, squared integrable random variable.
Iterate
\begin{equation}
 \label{e:Tsengaaamin}
 (\forall n\in\NN)\quad
\begin{array}{l}
 \left\lfloor
 \begin{array}{l}
 p_{n} = x_n-\gamma_n \sum_{i=1}^m \omega_i( L^{*}_iv_{i,n} + r_{n})\\
\operatorname{For}\;i=1,\ldots,m\\
\left\lfloor
\begin{array}{l}
 v_{i,n+1} = \prox^{(\tau_n/\gamma_n)U^{-1}_i}_{g^*_i}(v_{i,n}+\frac{\tau_n}{\gamma_n} U_i L_ip_{n})\\
\end{array}
\right.\\[2mm]
 x_{n+1} = x_{n} - \gamma_n\sum_{i=1}^m\omega_i(L^{*}_iv_{i,n+1} + r_{n})\\
 \end{array}
 \right.\\[2mm]
 \end{array}
 \end{equation}
\end{algorithm}

The following result is a direct consequence of Corollary~\ref{cor:comp}.
\begin{corollary}
In the setting of Problem~\ref{prob2}, suppose that 
\begin{equation}
0 \in \zer\Big(\nabla h + \sum_{i=1}^{m}\omega_i L_{i}^*\partial g_i(L_i\cdot) \Big)
\end{equation}
Let $(x_n,v_n)_{n\in\NN}$ be the sequence generated by Algorithm~\ref{algo:min}. 
 Suppose that the following conditions are satisfied for 
$\FF_n = \sigma((x_{k},(v_{i,k})_{1\leq i\leq m})_{0\leq k\leq n})$
 \begin{enumerate}
 \item there exists $\tau\in\left]0,+\infty\right[$ such that  $\sup_{n\in\NN}\tau_n\leq \tau$ and $(\forall i\in \{1,\ldots,m\})\quad $ $(\tau U_i)^{-1}- L_iL_{i}^*$ is positive definite.
 \item $\gamma_0\in\left]0,\mu\right[$ and $\inf\gamma_n>0$.
 \item $ \E[r_{n}|\FF_n] = \nabla h(x_{n})$.
 \item $\sum_{n\in\NN} \E[\|r_{n}-\nabla h(x_{n})\|^2 |\FF_n] < +\infty $  \quad \text{$\mathsf{P}$-a.s.}
 \end{enumerate}
 Then the following hold for some random vector $\overline{x}$ solving \eqref{e:primalmin} and some
 random vector $(\overline{v}_1,\ldots,\overline{v}_m)$ solving \eqref{e:dualmin} a.s.
 \begin{enumerate}
\item $(x_{n})_{n\in\NN}$ converges weakly to $\overline{x}$ almost surely. 
\item For every $i\in \{1,\ldots,m\}$ $(v_{i,n})_{n\in\NN}$ converges weakly to $\overline{v}_i$ a.s.
 \item $\sum_{n\in\NN}\|\nabla h(x_n)-\nabla h(x)\|^2 < +\infty$ a.s.
\end{enumerate}
\end{corollary}

\begin{corollary}
In the setting of Problem~\ref{prob2}, let $(x_n,v_n)_{n\in\NN}$ be the sequence generated by Algorithm~\ref{algo:min}. 
For every $N\in\NN$, define
\begin{align}
\nonumber
\hspace{-3cm}\tilde{x}_{N} &=  \Big(\sum_{n=0}^N\gamma_n\Big)^{-1}\sum_{n=0}^N \gamma_n x_{n+1}\\
\hspace{-3cm}\nonumber (\forall i\in\{1,\ldots,m\})\quad \tilde{v}_{i,N} &=  \Big(\sum_{n=0}^N\gamma_n\Big)^{-1}\sum_{n=0}^N \gamma_n v_{n+1}.
\end{align}
Suppose  that the following conditions are satisfied for $\FF_n = \sigma((x_k,v_k)_{0\leq k\leq n})$
\begin{enumerate}
\item 
$\gamma_{0}\in\left]0,\mu\right[$. 
\item
there exists $\tau\in\left]0,+\infty\right[$ such that  $\sup_{n\in\NN}\tau_n\leq \tau$ and, for every $i\in\{1,\ldots,m\}$,  the operator $(\tau U_i)^{-1}-LP_VL^*$ is positive semidefinite.
\item  $(\forall n\in\NN)\; \E[r_n |\FF_n] = \nabla h(x_n)$.
\item $c_0=\sum_{n\in\NN}\gamma_{n}^2 \E[\|r_{n}-\nabla h(x_{n})\|^2] < \infty$.
\end{enumerate}
Let $x\in \HH$ and let $\vv=(v_i)_{1\leq i\leq m}\in\cart_{i=1}^m\dom g^*_i$.
Set $c(x,\vv)= \E[\| x_0-x\|^2 +\gamma_0^2 \sum_{i=1}^m \omega_i\|v_{i,0}-v_{i}\|_{(\tau_0U_i^{-1}-LP_VL^*)}^2]
+ 2((\gamma_0\tau\|U\|)^{1/2} \|L\|^2+1)c_0.$ 
Then 
\begin{equation}
\label{e:res1b}
\E[H(\tilde{x}_N,(v_1,\ldots,v_m)) - H(x,(\tilde{v}_{1,N},\ldots,\tilde{v}_{m,N})) ]\leq \frac{c(x,\vv)}{2}\Big(\sum_{n=0}^N\gamma_n\Big)^{-1}.
\end{equation}
\end{corollary}

{\small {\bf Acknowledgments.} 
This material is based upon work supported by the Center for Brains, Minds and Machines (CBMM), funded by NSF STC award CCF-1231216. L. R. acknowledges the financial support of the Italian Ministry of Education, University and Research FIRB project RBFR12M3AC. S. V. is member of the Gruppo Nazionale per l?Analisi Matematica, la Probabilit\`a e le loro Applicazioni (GNAMPA) of the Istituto Nazionale di Alta Matematica (INdAM).
}

\end{document}